\documentclass[11pt]{amsart}
\usepackage{amsmath,amssymb,latexsym,tikz}
\usepackage{mathrsfs,amsfonts}
\usepackage[colorlinks=true,urlcolor=blue,
citecolor=red,linkcolor=blue,linktocpage,pdfpagelabels,
bookmarksnumbered,bookmarksopen]{hyperref}
\usepackage[english]{babel}

\usepackage[left=3cm,right=3cm,top=3cm,bottom=3cm]{geometry}

\usepackage[hyperpageref]{backref}

\usepackage{mathrsfs}
\usepackage{lmodern}

\numberwithin{equation}{section}

\newtheorem{theorem}{Theorem}[section]
\newtheorem{lemma}[theorem]{Lemma}

\newtheorem{prop}[theorem]{Proposition}
\newtheorem{corollary}[theorem]{Corollary}
\theoremstyle{definition}

\newtheorem{remark}[theorem]{Remark}

\newcommand{\N}{{\mathbb N}}

\newcommand{\R}{{\mathbb R}}

\newcommand{\eps}{\varepsilon}

\newcommand{\cal}{\mathcal}
\def\Xint#1{\mathchoice
{\XXint\displaystyle\textstyle{#1}}%
{\XXint\textstyle\scriptstyle{#1}}%
{\XXint\scriptstyle\scriptscriptstyle{#1}}%
{\XXint\scriptscriptstyle\scriptscriptstyle{#1}}%
\!\int}
\def\XXint#1#2#3{{\setbox0=\hbox{$#1{#2#3}{\int}$ }
\vcenter{\hbox{$#2#3$ }}\kern-.6\wd0}}

\def\dashint{\Xint-}

\title[Optimal elliptic regularity: local vs nonlocal]{Optimal elliptic regularity: a comparison between local and nonlocal equations}

\author[S.  Mosconi]{S. Mosconi}
\address[S. Mosconi]{Dipartimento di Matematica e Informatica
\newline\indent
Universit\`a degli Studi di Catania
\newline\indent
Viale A. Doria 6 I-95125 Catania, Italy}

\email{mosconi@dmi.unict.it}

\begin{document}
\begin{abstract}
Given $L\geq 1$, we discuss the problem of determining the highest $\alpha=\alpha(L)$ such that any solution to a homogeneous elliptic equation in divergence form with ellipticity ratio bounded by $L$ is in $C^\alpha_{\rm loc}$. This problem can be formulated both in the classical and non-local framework. In the classical case it is known that $\alpha(L)\gtrsim {\rm exp}(-CL^\beta)$, for some $C, \beta\geq 1$ depending on the dimension $N\geq 3$. We show that in the non-local case, $\alpha(L)\gtrsim L^{-1-\delta}$ for all $\delta>0$.
\end{abstract}
\maketitle
\section{Introduction}

The aim of this short note is to highlight some analogies and differences between the H\"older regularity theory for elliptic equations in divergence form with measurable coefficients in the classical and non-local setting respectively. To this end, let us recall some well known results. 

Suppose that $A(x)=(a_{ij}(x))$ is an $N\times N$ symmetric matrix with measurable coefficients, satisfying for some $0<\lambda\leq \Lambda $ the ellipticity condition
\begin{equation}
\label{ellcond}
\lambda |\xi|^2\leq \sum_{i, j=1}^N a_{ij}(x)\xi_i\xi_j\leq \Lambda |\xi|^2,\qquad \forall x\in B_1, \ \xi\in \R^N.
\end{equation}
If $u\in W^{1,2}(B_1)$, $B_1$ being the unit ball centered at $0$, weakly solves 
\begin{equation}
\label{div}
{\rm div}(A(x)\nabla u)=0
\end{equation}
in $B_1$, then, the De Giorgi-Nash-Moser theorem states that $u$ is H\"older continuous in $B_{1/2}$. A well known open problem in this framework is to determine the so-called {\em best H\"older exponent}\footnote{Technically, is should be called {\em worst H\"older exponent}, however the adjective "best" is more used in the literature.}  for a given ellipticity ratio $L=\Lambda /\lambda$. To precisely state the problem, let 
\[
{\cal H}_L=\big\{u\in W^{1,2}(B_1): \exists\  a_{ij} \text{ satisfying \eqref{ellcond} with $\frac{\Lambda}{\lambda}\leq L$, s.t. $u$ solves \eqref{div}}\big\}.
\] 
By the linearity and homogeneity of the equation, we can multiply \eqref{div} by $\lambda^{-1}$, so that we can assume $\lambda=1$ in \eqref{ellcond} (and then $\Lambda=L$).
The {\em best H\"older exponent} for solutions of elliptic equations with a given ellipticity ratio $L$ is
\[
\bar{\alpha}(N, L):=\sup\big\{\alpha: {\cal H}_L\subseteq  C^\alpha_{\rm loc}(B_{1})\big\}.
\]
In \cite{PS}, the following is proved:
\[
\bar{\alpha}(2, L)=\frac{1}{\sqrt{L}},\qquad \bar{\alpha}(N, L)< \frac{N-1}{N-2}\frac{1}{L}\quad \text{if $N\geq 3$}.
\]
It is conjectured that $\bar{\alpha}(N, L)=C(N)/L$, but the problem of determining $\bar{\alpha}(N, L)$ for $N\geq 3$ is open up to now. The best available lower bound on $\bar\alpha$ can be found keeping trace of the constants in the various proofs of the De Giorgi-Nash-Moser theorem: as it turns out (in all three proofs!), one finds for $N\geq 3$
\begin{equation}
\label{alpha}
\bar\alpha(N, L)\geq \exp (-C\, L^\beta), \qquad C=C(N), \ \beta=\beta(N)>1,
\end{equation}
which is of course (exponentially) far from the known upper bound.
\smallskip

Let us mention that after the work of Moser \cite{MH}, many H\"older regularity results are nowadays  derived from the Harnack inequality. Indeed, any {\em positive} solution of \eqref{div} in a ball $B_r$ satisfies
\[
\sup_{B_{r/2}}u\leq c_{H}\inf_{B_{r/2}} u
\]
for a constant $c_H$ independent of $r$.
By a standard argument, this implies that any solution of \eqref{div} satisfies
\[
{\rm osc}(u, 2^{-(n+1)})\leq \frac{c_H-1}{c_H+1}{\rm osc}(u, 2^{-n}),\qquad n\geq 1
\]
(here ${\rm osc}(u, r)=\sup_{B_{r}}u-\inf_{B_{r}} u$), which in turn gives the oscillation estimate 
\[
{\rm osc}(u, r)\leq C\, {\rm osc}(u, 1) \, r^{\alpha_H},\qquad \alpha_H\simeq \frac{1}{c_H}.
\]
In other words, the {\em best }(scale invariant) {\em Harnack constant}  for \eqref{div} bounds the inverse of the best H\"older exponent. Unfortunately, it is well known that the best Harnack constant is in fact exponential in the ellipticity ratio: for example, the function
$u(x, y)=e^{\sqrt{L}x} \cos y$ satisfies \eqref{div} for 
\[
A=
\begin{pmatrix}
1&0\\
0& L
\end{pmatrix}
\]
but for any $r\leq 1$ it holds
\[
\sup_{B_r} u=e^{\sqrt{L}r},\quad \inf_{B_r} u\leq e^{-\sqrt{L} r},\qquad c_H=\sup_{0<r<1}\frac{\sup_{B_r} u}{\inf_{B_r} u}\geq e^{2\sqrt{L}}.
\]
This example shows that one cannot get rid of the exponential in \eqref{alpha} by proving a scale invariant Harnack inequality
\footnote{By scale invariant we mean with a constant independent on $r$ for {\em small} $r$.} and it is not so surprising that  the methods in \cite{DG} \cite{N} and \cite{M} give the bound \eqref{alpha} since each technique actually proves a Harnack inequality as well (see \cite{DBT}, \cite{FS} and \cite{MH}, respectively). 

Our main result shows that the nonlocal analogue of \eqref{div} seems to enjoy better properties than \eqref{div} itself. By a non-local analogue of solutions to \eqref{div}  we mean local minimizers of the non-local functional
\begin{equation}
\label{DirF}
\int_{\R^{2N}}a(x, y)\frac{(u(x)-u(y))^2}{|x-y|^{N+2s}}\, dx\, dy,
\end{equation}
as opposed to local minimizers of 
\begin{equation}
\label{Dir}
\int_{\Omega} \big| A^{1/2}\nabla u\big|^2 dx
\end{equation}
giving rise to solutions of \eqref{div}.
Let us recall in this regard that, for suitable $a(x, y)$, \eqref{DirF} approximates \eqref{Dir} in a suitable sense when $s\to 1$. This can be seen, at least for what concerns pointwise convergence, through \cite[Corollary 1]{BBM} (see also \cite[Theorem 3.1]{BPS}): if $0< \lambda\leq a(x)\leq \Lambda$, then
\[
\lim_{s\to 1}\int_{\R^{2N}}(1-s)a(x)\frac{(u(x)-u(y))^2}{|x-y|^{N+2s}}\, dx= C(N)\int_{\Omega} a(x)|\nabla u|^2,\qquad \forall u\in W^{1,2}_0(\Omega).
\]
In turn, isotropic functionals (i.e. \eqref{Dir} with $a_{ij}(x)=a(x)\delta_{ij}$) are dense with respect to Mosco-convergence in the set of general elliptic functionals \eqref{Dir} by \cite{MT} (see also \cite{CS} for a more general density result and references). 

Notice that, when looking at local minimizers of \eqref{DirF}, we can always suppose that $a(x, y)$ is symmetric, since 
\[
\int_{\R^{2N}}a(x, y)\frac{(u(x)-u(y))^2}{|x-y|^{N+2s}}\, dx\, dy=\int_{\R^{2N}}\frac{a(x, y)+a(y, x)}{2}\frac{(u(x)-u(y))^2}{|x-y|^{N+2s}}\, dx\, dy.
\]
In this framework, the ellipticity condition \eqref{ellcond} reads $\lambda\leq a(x, y)\leq \Lambda$, which implies the same condition of the symmetrized coefficient
\[
k(x, y)=\frac{a(x, y)+a(y, x)}{2}.
\]
Local minimizers of \eqref{DirF} in, say, $B_1$, satisfy
\begin{equation}
\label{eq}
\int_{\R^{2N}}(u(x)-u(y))(\varphi(x)-\varphi(y))\, \frac{k(x, y)\, dx\, dy}{|x-y|^{N+2s}}=0\qquad \forall \varphi\in W^{s,2}_0(B_1),
\end{equation}
where
\begin{equation}
\label{w0}
W^{s,2}_0(B_1)=\left\{v\in L^1(\R^N): v\equiv 0\ \text{in $\R^N\setminus B_1$}, \ [v]_s<+\infty\right\},
\end{equation}
and
\[
[v]_s=\left(\int_{\R^{2N}}\frac{(v(x)-v(y))^2}{|x-y|^{N+2s}}\, dx\, dy\right)^{1/2}.
\]
Normalizing, we will suppose henceforth
\begin{equation}
\label{condk}
1\leq k(x, y)=k(y, x)\leq L=\frac \Lambda \lambda \quad \text{for a.e. $x$, $y$}.
\end{equation}

We can now state our main result, in a simplified version (for the full result see Theorem \ref{mt}).

\begin{theorem}
\label{main}
Let $u\in L^\infty(\R^N)$ be such that $[u]_s<+\infty$ and \eqref{eq} holds under condition \eqref{condk} on $k$. Then for any $\delta>0$, $u\in C^{\alpha_\delta}(B_{1/2})$ with $\alpha_\delta=a_\delta/L^{1+\delta}$.
\end{theorem}

Let us notice that the De Giorgi-Nash-Moser regularity theory for equations like \eqref{eq} has already been developed in \cite{K0}, \cite{DKP}. The method of proof in these papers is usually a (nontrivial) modification of the De Giorgi-Moser approach, and therefore provides a much smaller H\"older exponent than ours, namely, the one given in \eqref{alpha}. 
The main point of our result is therefore the improvement, in the non-local case, of the (inverse of the) best H\"older exponent from exponential in $L$ to almost linear in $L$. This is mainly due to a weak Harnack inequality which is proved, much in the spirit of \cite{IMS}, \cite{IMS2}, through a non-local barrier argument, a method which avoids the usual De Giorgi-Moser arguments. We remark that the obtained H\"older exponent {\em blows down} as $s\to 1^-$, i.e.  our estimates are exclusively  non-local.

Let us finally mention that, starting from the seminal work of Bass and Levine \cite{BL}, a huge literature has grown around the regularity theory for some related non-local equations, which, very roughly speaking, may be seen as the non-local counterpart of {\em non-divergence form} elliptic equations, see \cite[Section 3.6]{S} for a discussion. The interested reader can consult the bibliographic references in \cite{SS}.

\medskip

The structure of the paper is as follows.
In section 2 we will describe the functional analytic framework we will work in. Section 3 is devoted to a lower bound of the torsion function, i.e., the solution $v$ of 
\[
\int_{\R^{2N}}(v(x)-v(y))(\varphi(x)-\varphi(y))\, \frac{k(x, y)\, dx\, dy}{|x-y|^{N+2s}}=\int_{\R^N} \varphi\, dx, \qquad \forall \varphi\in W^{s,2}_0(B_1).
\]

Despite $v$ being, strictly speaking, a supersolution of \eqref{eq}, the non-local nature of the equation allows to use it as a {\em lower barrier} to \eqref{eq}, when suitably modified away from $B_1$. This idea will be realized in section 4, giving a weak Harnack inequality and, eventually, the claimed H\"older exponent.

\section{Preliminaries}

In all the paper $N\geq 1$, $s\in \ ]0, 1[$, $L\geq 1$ will be fixed. We also fix $k(x, y)$ satisfying \eqref{condk}, and for simplicity we let 
\begin{equation}
\label{mu}
d\mu=\frac{k(x, y)}{|x-y|^{N+2s}}\, dx\, dy.
\end{equation}
All constants, unless otherwise specified, will depend on $N$ and $s$ only,  and are allowed to change value from line to line keeping the same symbol, as long as this dependance is unaltered. When a constant depends also on some other parameter $\delta$, it will be denoted as, say, $C_\delta$, i.e. $C_\delta=C(N, s, \delta)$.

For any $A\subseteq \R^N$ we will set $A^c=\R^N\setminus A$ and $B_r\subseteq \R^N$ will denote the open ball of radius $r>0$, centered at $0$.  For simplicity, $\sup_A u$ and $\inf_A u$ will denote the essential supremum and infimum of $u$ on a measurable set $A$, and $\|u\|_\infty$ will be the $L^\infty(\R^N)$ of $u$. For any $a\in \R$, we also let $a_+=\max\{a, 0\}$, $a_-=\min\{a, 0\}$. 

Given any\footnote{For our purposes $\Omega$ will  always be a ball, so that delicate issues for the defined space when $\Omega$ is irregular/unbounded are not relevant here} open $\Omega\subseteq \R^N$ and measurable $u:\Omega\to \R$, we define the Gagliardo semi-norms
\[
[ u ]_{s, \Omega}:=\left(\int_{\Omega\times\Omega} \frac{(u(x)-u(y))^2}{|x-y|^{N+2s}}\, dx\, dy\right)^{1/2},\qquad [u]_s:=[u]_{s, \R^N}
\]
and the corresponding functional spaces 
\[
\widetilde{W}^{s, 2}(\Omega):=\left\{u:\exists\,  A\supset \overline{\Omega},  \text{ \, $A$ open, s.t. }\  [u]_{s, A}+\int_{A^c}\frac{|u(x)|}{1+|x|^{N+2s}}\, dx<+\infty\right\}.
\]
Clearly, as long as $u\in L^\infty(\R^N)$, the second condition in the previous definition is trivially satisfied.
The space $W^{s,2}_0(\Omega)$ is defined as in \eqref{w0}, substituting $B_1$ with $\Omega$, and its dual is denoted by $W^{-s, 2}(\Omega)$.
Suitably modifying \cite[Lemma 2.3]{IMS}, we see that for any $u\in \widetilde{W}^{s,2}(\Omega)$, the functional 
\[
W^{s,2}_0(\Omega)\ni \varphi\mapsto \int_{\R^{2N}}(u(x)-u(y))(\varphi(x)-\varphi(y))\, d\mu=:\langle {\cal K} u, \varphi\rangle
\]
is linear and continuous in $W^{s,2}_0(\Omega)$, where $\mu$ is defined in \eqref{mu} and $k$ fulfills \eqref{condk}. Therefore the notation used for ${\cal K}$ above is well posed when $\langle \ , \ \rangle$ denotes the duality bracket of $W^{s, 2}_0(\Omega)$ and ${\cal K}: \widetilde{W}^{s,2}(\Omega)\to W^{-s, 2}(\Omega)$ is clearly a linear operator. Using the linearity of the equation it is therefore possible to solve any Dirichlet problem of the form 
\[
\begin{cases}
{\cal K} u=0 &\text{in $\Omega$},\\
u=u_0&\text{in $\Omega^c$}
\end{cases}
\]
for any given $u_0\in \widetilde{W}^{s,2}(\Omega)$, simply minimizing the convex coercive functional
\[
W^{s,2}_0(\Omega)\ni v\mapsto \frac{1}{2}[v]_{s}^2+\langle {\cal K} u_0, v\rangle
\]
and letting $u=v+u_0\in \widetilde{W}^{s,2}(\Omega)$.
This setting also gives a meaning to inequalities such as ${\cal K} u\geq f$ in $\Omega$ for $u\in \widetilde{W}^{s,2}(\Omega)$, $f\in W^{-s, 2}(\Omega)$ by testing with nonnegative $\varphi\in W^{s,2}_0(\Omega)$. 

\medskip
We will make use of the following comparison principle, obtained by a slight modification of \cite[Proposition 2.10]{IMS}.

\begin{prop}{\em (Comparison Principle)}
Let $\Omega$ be bounded and $u, v\in \widetilde{W}^{s,2}(\Omega)$ be such that 
\[
\begin{cases}
\langle {\cal K} u, \varphi\rangle \leq \langle \cal K v, \varphi\rangle& \forall \varphi\in W^{s,2}_0(\Omega), \ \varphi\geq 0\\
u\leq v&\text{in $\Omega^c$}.
\end{cases}
\]
Then $u\leq v$ in $\R^N$.
\end{prop}

\section{Lower bound on the torsion function}

This section is devoted to obtain a lower bound, in terms of the ellipticity ratio $L$, for the torsion function, namely the solution to ${\cal K} u=1$ with Dirichlet boundary conditions. We do not know wether the following proposition holds with $\delta=0$, a fact that would eventually provide a H\"older exponent of order $1/L$ in Theorem \ref{main}

\begin{prop}
\label{tor}
Assume that \eqref{condk} holds and let $u$ weakly solve 
\begin{equation}
\label{torsion}
\begin{cases}
{\cal K}u=1&\text{in $B_r$},\\
u\equiv 0&\text{in $B_r^c$}.
\end{cases}
\end{equation}
Then for any $\delta>0$ there exists a constant $c_\delta=c(N, s, \delta)$ such that 
\begin{equation}
\label{lbtorsion}
\inf_{B_{r/2}} u\geq \frac{c_\delta\, r^{2s}}{L^{1+\delta}}.
\end{equation}
\end{prop}

\begin{proof}
First we show that by a scaling argument we can suppose $r=1$. Let $u$ solve \eqref{torsion} and define 
\[
u_{(r)}(x)=u(rx), \qquad k_{(r)}(x, y)=k(rx, ry)
\]
and ${\cal K}_{(r)}$ the corresponding linear operator.
Then by changing variables it follows that $u_{(r)}\in W^{s, 2}_0(B_1)$ satisfies ${\cal K}_{(r)}u_{(r)}=r^{2s}$, while clearly $k_{(r)}$ still satisfies the bounds \eqref{condk}. The linearity of the equation then gives that $u_{(r)}r^{-2s}$ satisfies \eqref{torsion} in $B_1$, which gives the general statement \eqref{lbtorsion}. Therefore we will suppose henceforth that $r=1$.

{\bf Step 1}: {\em Caccioppoli inequality for negative powers of $u$}.

By the weak minimum principle it holds $u\geq 0$ in $B_1$. We fix $\eps>0$ and let $u_\eps=u+\eps$. For any $\rho_2<\rho_1<1$, we choose a cutoff function 
\begin{equation}
\label{eta}
\eta\in C^\infty_c(B_{\rho_1}, [0, 1]),\qquad  \text{$\eta\equiv 1$ \ in $B_{\rho_2}$} \qquad |\nabla \eta|\leq \frac{C}{\rho_1-\rho_2}
\end{equation}
and for any $\beta>1$ we test the equation with $\eta^{\beta+1} u_\eps^{-\beta}$ obtaining
\[
\int_{\R^N} \frac{\eta^2}{u_\eps^\beta}\, dx=\int_{\R^{2N}}(u(x)-u(y))\left(\frac{\eta^{\beta+1}(x)}{u_\eps^\beta(x)}-\frac{\eta^{\beta+1}(y)}{u_\eps^{\beta}(y)}\right)\, d\mu.
\]
We employ \cite[Lemma 2.5]{K0}, to obtain the pointwise inequality
\[
\begin{split}
(u_\eps(x)-u_\eps(y))&\left(\frac{\eta^{\beta+1}(y)}{u_\eps^\beta(y)}-\frac{\eta^{\beta+1}(x)}{u_\eps^{\beta}(x)}\right)\geq \frac{\eta(x)\eta(y)}{\beta-1}\left((\frac{\eta(x)}{u_\eps(x)})^{\frac{\beta-1}{2}}-(\frac{\eta(y)}{u_\eps(y)})^{\frac{\beta-1}{2}}\right)^2\\
&\qquad\qquad  -4\beta\, (\eta(x)-\eta(y))^2\left((\frac{\eta(x)}{u_\eps(x)})^{\beta-1}+(\frac{\eta(y)}{u_\eps(y)})^{\beta-1}\right).
\end{split}
\]
which gives
\begin{equation}
\label{cac0}
\begin{split}
&\int_{\R^N} \frac{\eta^2}{u_\eps^\beta}\, dx+\int_{\R^{2N}}\frac{\eta(x)\eta(y)}{\beta-1}\left((\frac{\eta(x)}{u_\eps(x)})^{\frac{\beta-1}{2}}-(\frac{\eta(y)}{u_\eps(y)})^{\frac{\beta-1}{2}}\right)^2 d\mu \\
&\qquad \leq 4\beta\int_{\R^{2N}}(\eta(x)-\eta(y))^2\left((\frac{\eta(x)}{u_\eps(x)})^{\beta-1}+(\frac{\eta(y)}{u_\eps(y)})^{\beta-1}\right)\, d\mu.
\end{split}
\end{equation}
The right hand side can be bounded using the symmetry and structure of $\mu$ as
\[
\begin{split}
\int_{\R^{2N}}(\eta(x)-\eta(y))^2&\left[\left(\frac{\eta(x)}{u_\eps(x)}\right)^{\beta-1}+\left(\frac{\eta(y)}{u_\eps(y)}\right)^{\beta-1}\right] \,d\mu\\
&\qquad \leq   2\,L\int_{\R^{2N}}\frac{(\eta(x)-\eta(y))^2}{|x-y|^{N+2s}}\left(\frac{\eta(x)}{u_\eps(x)}\right)^{\beta-1}\! dxdy
\end{split}
\]
and estimating the integral in $y$ through \cite[Lemma 2.3]{MM}, which gives
\begin{equation}
\label{cac1}
\begin{split}
\int_{\R^{2N}}\frac{(\eta(x)-\eta(y))^2}{|x-y|^{N+2s}}\left(\frac{\eta(x)}{u_\eps(x)}\right)^{\beta-1} dx\, dy&\leq 
C\,{\rm Lip}(\eta)^{2s}\int_{\R^N}\left(\frac{\eta}{u_\eps}\right)^{\beta-1} dx\\
&\leq \frac{C}{(\rho_1-\rho_2)^{2s}}\int_{B_{\rho_1}}u_\eps^{1-\beta} dx.
\end{split}
\end{equation}
We bound the left hand side of \eqref{cac0} from below using Sobolev-Poincar\'e inequality as follows:
\begin{equation}
\label{cac2}
\begin{split}
\int_{\R^N}& \frac{\eta^2}{u_\eps^\beta}\, dx+\int_{\R^{2N}}\frac{\eta(x)\eta(y)}{\beta-1}\left[\left(\frac{\eta(x)}{u_\eps(x)}\right)^{\frac{\beta-1}{2}}-\left(\frac{\eta(y)}{u_\eps(y)}\right)^{\frac{\beta-1}{2}}\right]^2 d\mu\\
&\geq \int_{B_{\rho_2}}u_\eps^{-\beta}\, dx+\beta^{-1}\int_{B_{\rho_2}\times B_{\rho_2}}\left(u_\eps^{\frac{1-\beta}{2}}(x)-u_\eps^{\frac{1-\beta}{2}}(y)\right)^2 \frac{dx\, dy}{|x-y|^{N+2s}}\\
&\geq \int_{B_{\rho_2}}u_\eps^{-\beta}\, dx+C\beta^{-1}\left(\int_{B_{\rho_2}}u_\eps^{2^*\frac{1-\beta}{2}}\, dx\right)^{\frac{2}{2^*}}- \frac{C\beta^{-1}}{\rho_2^{2s}}\int_{B_{\rho_2}}u_\eps^{1-\beta}\, dx.
\end{split}
\end{equation}
Inserting \eqref{cac1} and \eqref{cac2} into \eqref{cac0} and rearranging, we obtain the following Caccioppoli inequality:
\begin{equation}
\label{cac}
\int_{B_{\rho_2}}\!u_\eps^{-\beta}dx+\beta^{-1}\left(\int_{B_{\rho_2}}\!u_\eps^{2^*\frac{1-\beta}{2}} dx\right)^{\frac{2}{2^*}}\leq
CL \left(\frac{\beta}{(\rho_2-\rho_1)^{2s}}+\frac{\beta^{-1}}{ \rho_1^{2s}}\right)\int_{B_{\rho_1}}\!u_\eps^{1-\beta} dx.
\end{equation}

{\bf Step 2}: {\em The iteration}.

We first iterate \eqref{cac} neglecting the second term on the left. We fix a large $m\in \N$ to be chosen later and let
\[
\rho_i=\frac{2}{3}+\frac{1}{6}\frac{m-i}{m}, \qquad i=1, \dots, m.
\]
 From \eqref{cac} we infer
\[
\int_{B_{\rho_m}}\!u_\eps^{-m}dx\leq L\, C\int_{B_{\rho_{m-1}}}\!u_\eps^{-m+1} dx\leq \dots
\leq L^{m-1} C_m\int_{B_{\rho_1}}\!u_\eps^{-1} dx,
\]
where the constant on the left hand side satisfies
\[
C_m\leq C^m m!\, m^{2sm}, 
\]
thus giving, raising the previous estimate to the $1/m$-power,
\begin{equation}
\label{it1}
\|u_\eps^{-1}\|_{L^m(B_{2/3})}\leq C_m^{\frac 1 m}L^{1-\frac{1}{m}}\,\|u_\eps^{-1}\|_{L^1(B_{5/6})}^{1/m}.
\end{equation}
Next, we follow the standard Moser iteration starting from the power $m$ and ignoring the first term on the left of \eqref{cac}.  We let $\gamma=\frac{2^*}{2}>1$,
\[
\rho_i=\frac{1}{2}+\frac{2^{-i}}{6}, \qquad i\geq 0
\]
and use \eqref{cac} with $\beta-1=t>1$ to obtain
\[
\begin{split}
\left(\int_{B_{\rho_{i+1}}}u_\eps^{-\gamma t}\, dx\right)^{\frac{2}{2^*}}&\leq C\, t^2L \left(\frac{1}{(\rho_{i+1}-\rho_i)^{2s}}+\frac{1}{ \rho_i^{2s}}\right)\int_{B_{\rho_i}}u_\eps^{-t}\, dx\\
&\leq \frac{C\, 4^{is}t^2L}{r^{2s}}\int_{B_{\rho_i}}u_\eps^{-t}\, dx.
\end{split}
\]
The latter can be rewritten as
\[
\|u_\eps^{-1}\|_{L^{\gamma t}(B_{r_{i+1}})}\leq \left(\frac{C\, 4^{is}t^2L}{r^{2s}}\right)^{1/t}\|u_\eps^{-1}\|_{L^t(B_{r_i})},
\]
so that, for $t_i=\gamma^{i}m$, $i\geq 0$, we have
\[
\|u_\eps^{-1}\|_{L^{t_{i+1}}(B_{r_{i+1}})}\leq \left(C\, 4^{is}t_i^2L\right)^{1/t_i}\|u_\eps^{-1}\|_{L^{t_i}(B_{r_i})}.
\]
This can be iterated estimating the infinite product (we precisely compute only the relevant one) as 
\[
\begin{split}
\sup_{B_{1/2}} u_\eps^{-1}&\leq \lim_{i\to +\infty}\|u_\eps^{-1}\|_{L^{t_{i+1}}(B_{r_{i+1}})}\leq C\prod_{i=0}^{+\infty}L^{1/t_i}\|u_\eps^{-1}\|_{L^m(B_{2/3})}\\
&\leq C\,L^{\frac{1}{m}\frac{\gamma}{\gamma-1}}\|u_\eps^{-1}\|_{L^m(B_{2/3})}.
\end{split}
\]
Applying \eqref{it1} on the right hand side, we obtain
\begin{equation}
\label{it2}
\sup_{B_{1/2}} u_\eps^{-1}\leq C_mL^{1-\frac 1 m+\frac{1}{m}\frac{\gamma}{\gamma-1}}\|u_\eps^{-1}\|_{L^1(B_{5/6})}^{1/m}.
\end{equation}

{\bf  Step 3}: {\em Conclusion.}

Finally we estimate the right hand side of \eqref{it2} testing the equation with $\eta^2/u_\eps$, for $\eta$ as in \eqref{eta}, with $\rho_2=5/6$ and $\rho_1=1$. The following elementary inequality holds for $a, b>0$, $c, d\geq 0$:
\[
(a-b)\left(\frac{c^2}{a}-\frac{d^2}{b}\right)=c^2+d^2-\left(\frac{b}{a}c^2+\frac{a}{b}d^2\right)\leq (c-d)^2.
\]
The latter inequality immediately yields
\[
\int_{\R^N}\frac{\eta^2}{u_\eps}\, dx=\int_{\R^{2N}}\!(u(x)-u(y))\left(\frac{\eta^2(x)}{u_\eps(x)}-\frac{\eta^2(y)}{u_\eps(y)}\right)d\mu\leq \int_{\R^{2N}}\!(\eta(x)-\eta(y))^2d\mu\leq CL,
\]
so that 
\begin{equation}
\label{hja}
\int_{B_{5/6}}u_\eps^{-1}\, dx\leq C\, L.
\end{equation}
Inserting into \eqref{it2} and letting $\eps\to 0$ gives
\[
\sup_{B_{1/2}} u^{-1}\leq C_mL^{1+\frac{1}{m}\frac{1}{\gamma-1}},
\]
which, rearranged, gives \eqref{lbtorsion} being $m$ arbitrarily large. 
\end{proof}

\begin{remark}
The previous proof is completely local in nature and a similar estimate holds true for the torsion function associated to elliptic equations with measurable coefficients.  In this more classical framework it may be interesting to know wether an optimal bound from below is true without the $\delta$ mentioned in the proposition. More precisely, if $A=(a_{i j})$ is measurable and  symmetric and $u$ is the weak solution of  
\[
\begin{cases}
{\rm div}(A\nabla u)=1 &\text{in $B_1$},\\
u=0&\text{in $\partial B_1$},
\end{cases}
\qquad\text{with}\qquad  |\xi|^2\leq a_{i  j}\xi_i\xi_j\leq L |\xi|^2,
\]
one may ask if there exists a constant $c=c(N)$ such that
\[
\inf_{B_{1/2}}u\geq \frac{c(N)}{L}.
\]
\end{remark}

\begin{remark}
\label{rems}
The dependance of $c_\delta$ on $s$ as $s\to 1^-$ can be singled out as follows. In estimate \eqref{cac1} one has (see the dependance in \cite[Lemma 2.3]{MM}) $C=C(N)/(1-s)$. Hence in \eqref{it1} we have $C_m=\tilde{C}_m(N)/(1-s)^{1-\frac 1 m}$. Similarly, in \eqref{hja}, one has $C=C(N)/(1-s)$.  Moreover, due to \cite{BBM, BPS} the Sobolev-Poincar\'e embedding can be written as 
\[
 c(N) \|v\|_{L^{2^*}(B_r)}^2\leq (1-s)[v]_s^2+\frac{C(N)}{r^{2s}}\|v\|_{L^2(B_r)}^2
\]
so that the second part of the iteration is independent of $s$. Following the proof one eventually gets $c_\delta=(1-s)\tilde{c}(N, \delta)$. This is consistent with the case $k(x, y)\equiv 1$ and ${\cal K}=(-\Delta)^s$: indeed, if $u\in W^{s,2}_0(B_r)$ solves $(-\Delta)^s u_s=1$, then $(1-s)u_s\to u$ as $s\to 1^-$, where $u$ is the classical torsion function for the Laplacian.
\end{remark}
\section{Proof of the main result}

We start with the core nonlocal estimate providing a weak Harnack inequality with best constant of order $L^{1+\delta}$, $\delta>0$. The fact that we allow arbitrary $K\geq 0$ in the statement is essential in order to deal with tail terms later in Corollary \ref{cor}, where we localize the condition $u\geq 0$ in $\R^N$.

\begin{lemma}[Weak Harnack inequality]
Let \eqref{condk} holds and $u\in\widetilde W^{s,2}(B_{R})$ satisfies
\begin{equation}
\label{sups}
\begin{cases}
{\cal K}u\geq -K& \text{weakly in $B_{R}$} \\
u\geq 0 & \text{in $\R^N$,}
\end{cases}
\qquad R>0, \ K\geq 0.
\end{equation}
Then,  for any $\delta>0$ there exists $\sigma_\delta=\sigma(N, s, \delta)\in (0, 1)$ such that 
\begin{equation}
\label{temp}
\inf_{B_{R/2}}u\geq\frac{\sigma_\delta}{L^{1+\delta}}\left(\dashint_{B_{2R}\setminus B_{3R/2}} u\,dx-C K R^{2s}\right),
\end{equation}
where $C=C(N)$.
\end{lemma}

\begin{proof} By the scaling argument described at the beginning of Proposition \ref{tor}, we can assume $R=1$. Let $v$ solve \eqref{torsion} in $B_1$, choose $\lambda\in(0,1)$ (to be determined later) and set
\[
M=\dashint_{B_2\setminus B_{3/2}} u\,dx, \qquad  w=\lambda M v+\chi_{B_2\setminus B_{3/2}}u.
\]
Since $\chi_{B_2\setminus B_{3/2}}u\in \widetilde{W}^{s,2}(B_1)$, we can compute ${\cal K}(\chi_{B_2\setminus B_{3/2}}u)$, setting for simplicity $\chi=\chi_{B_2\setminus B_{3/2}}$. For any $\varphi\in W^{s,2}_0(B_1)$, $\varphi\geq 0$, it holds
\[
\int_{\R^{2N}}(\chi(x) u(x)-\chi(y) u(y))(\varphi(x)-\varphi(y))\, d\mu=
-2\int_{B_1\times B_1^c}\varphi(x)\chi(y)u(y)\, d\mu
\]
where we used the symmetry of $\mu$ and the fact that ${\rm supp}(\chi u)\subseteq B_2\setminus B_{3/2}^c$. Since both $u$ and $\varphi$ are nonnegative and $|x-y|\leq 7/2$ for $x\in B_1$ and $y\in B_{3/2}^c$, we have
\[
\int_{B_1\times B_1^c}\varphi(x)\chi(y)u(y)\, d\mu =\int_{B_1\times (B_2\setminus B_{3/2})}\varphi(x)u(y)\, d\mu\geq \theta\,\dashint_{B_2\setminus B_{3/2}} u\,dy\int_{B_1}\varphi\, dx,
\]
for some $\theta=\theta(N)<1$, which implies
\[
{\cal K}(\chi_{B_2\setminus B_{3/2}}u)\leq - 2\theta\, M\qquad \text{weakly in $B_1$}.
\]
Therefore, by linearity and ${\cal K}v=1$,
\[
{\cal K} w\leq M(\lambda - 2\theta)\qquad \text{weakly in $B_1$}.
\]
We choose $\lambda=\theta$ in the definition of $w$ and, for  $\delta>0$, let $c_\delta\in \ ]0, 1]$ be the constant given in Proposition \ref{tor}. We claim that \eqref{temp} holds with $\sigma_\delta=\theta \, c_\delta$, $C=1/\theta$.
Indeed,  if $K=0$ in \eqref{sups}, we observe that
 \[
 {\cal K} w\leq - \theta \,M\leq 0\leq {\cal K}u\quad \text{weakly in $B_1$},\qquad  \text{$w\leq u$ in $B_1^c$},
 \]
 so that the weak comparison principle implies $u\geq w$ in $B_1$. In particular
 \begin{equation}
 \label{concl}
 \inf_{B_{1/2}}u\geq \theta M \inf_{B_{1/2}} v\geq  \frac{\theta\, c_\delta}{L^{1+\delta}}\,\dashint_{B_2\setminus B_{3/2}} u\,dx
 \end{equation}
 by Proposition \ref{tor}.  If $K>0$, we can use linearity to reduce to $K=1$ and prove
 \begin{equation}
 \label{kl}
 \inf_{B_{1/2}}u\geq \frac{\theta\, c_\delta}{L^{1+\delta}}\left( M-\frac{1}{\theta}\right).
 \end{equation}
 Clearly we can suppose that $M\geq 1/\theta$, since otherwise the right hand side of \eqref{kl} is negative and the inequality is trivially satisfied since $u\geq 0$ in $B_1$ by assumption. In this case 
 \[
 {\cal K} w\leq -\theta\,M\leq -1\leq {\cal K} u\quad \text{weakly in $B_1$}
 \]
 and the (stronger) conclusion \eqref{concl} follows as before by comparison.
 
\end{proof}

We now weaken the global condition  $u\geq 0$ in $\R^N$ assumed in the previous lemma. To this end we define 
\[
{\rm Tail}(u, R)=R^{2s}\int_{B_R^c}\frac{|u|}{|y|^{N+2s}}\, dy.
\]

\begin{corollary}
\label{cor}
Let \eqref{condk} holds. For any $\delta>0$ there exists $\sigma_\delta=\sigma(N, s, \delta)\in \ ]0, 1[$, $C_\delta=C(N, s, \delta)\geq 0$ such that any $u\in\widetilde W^{s,2}(B_{R})$ satisfying
\[
\begin{cases}
{\cal K}u=0& \text{weakly in $B_{R}$} \\
u\geq 0 & \text{in $B_{2R}$,}
\end{cases}
\qquad R>0, 
\]
 fulfills
\begin{equation}
\label{temp2}
\inf_{B_{R/2}}u\geq\frac{\sigma_\delta}{L^{1+\delta}}\,\dashint_{B_{2R}\setminus B_{3R/2}} u\,dx -\frac{C_\delta}{L^\delta}\, {\rm Tail}(u_-, 2R).
\end{equation}
\end{corollary}

\begin{proof}
We bound ${\cal K} u_-$ in $B_R$ as follows: since $u_-\equiv 0$ in $B_{2R}$, for any $\varphi\in W^{s,2}_0(B_R)$, $\varphi\geq 0$, it holds
\[
\begin{split}
\int_{\R^{2N}}(u_-(x)-u_-(y))(\varphi(x)-\varphi(y))\, d\mu&=-2\int_{B_R\times B_{2R}^c}\varphi(x)u_-(y)\, d\mu\\
&\leq L\int_{B_R\times B_{2R}^c}\frac{\varphi(x)|u_-(y)|}{|x-y|^{N+2s}}\, dx\, dy.
\end{split}
\]
Noting that $|x-y|\geq |y|/2$ for $x\in B_R$ and $y\in B_{2R}$, we see that
\[
{\cal K}(u_-)\leq CL \int_{B_{2R}^c}\frac{|u_-|}{|y|^{N+2s}}\, dy \qquad \text{weakly in $B_R$}
\]
for some $C=C(N)$. Thus, since $u=u_++u_-$, it holds
\[
{\cal K}(u_+)=-{\cal K} u_-\geq -CL \int_{B_{2R}^c}\frac{|u_-|}{|y|^{N+2s}}\, dy
\]
and applying \eqref{temp} to $u_+$ gives the claim.

\end{proof}

Now we are ready to prove the main result. The condition $u\in L^\infty(\R^N)$ is assumed only for simplicity, allowing the bound ${\rm Tail}(u, R)\leq C\|u\|_\infty$.
\begin{theorem}
\label{mt}
Let \eqref{condk} hold and $u\in \widetilde{W}^{s,2}(B_1)\cap L^\infty(\R^N)$ weakly solve ${\cal K} u=0$ in $B_1$. Then, for any $\delta>0$ there exist $C_\delta\geq 0$, $a_\delta\in \ ]0, 2s[$ such that 
\[
{\rm osc}(u, r)\leq C_\delta L \, \|u\|_{\infty}\, r^{a_\delta/L^{1+\delta}}, \qquad r\leq 1
\]

\end{theorem}

\begin{proof}
Let $\delta>0$ be fixed and, since the equation is linear, suppose also that $\|u\|_\infty=1$.
By the monotonicity of $r\mapsto {\rm osc}(u, r)$, it suffices to prove the estimate for the radii $r_i=4^{-i}$,  $i\geq 0$. To this end we set for simplicity
\[
B_i=B_{r_i},\qquad A_i=B_{r_i}\setminus B_{3r_i/4},
\]
 and  claim by induction that there exist $\alpha>0$, $k>0$, and corresponding sequences $\{M_i\}_{i}$, $\{m_i\}_i$ (the dependance of the latters on the parameters will be omitted) such that, for all $i\geq 0$, it holds
\begin{equation}
\label{ind}
M_i\geq \sup_{B_i}u\geq \inf_{B_i} u\geq m_i,\quad m_i\leq m_{i+1}\leq M_{i+1}\leq M_i,\quad M_i-m_i= k\,r_i^{\alpha}.
 \end{equation}
 The dependance of  $\alpha$ and $k$ on $L$ and $\delta$ will be determined later through various conditions. Since $\|u\|_{\infty}=1$, the inductive basis will be satisfied for 
 \begin{equation}
 \label{cond1}
 M_{0}= - m_{0}=k,\qquad k\geq 2.
 \end{equation}
 Suppose that $\{M_i\}$ and $\{m_i\}$ are chosen for $0\leq i\leq n$. 
Notice that $M_n-u$ and $u-m_n$ both satisfy the hypotheses of Corollary \ref{cor} for $R=r_n/2$, hence 
 \[
 \begin{split}
& \inf_{B_{n+1}} (M_n-u)\geq \frac{\sigma_\delta}{L^{1+\delta}}\dashint_{A_n}(M_n-u)\, dx -\frac{C_\delta}{L^\delta}\,{\rm Tail}((M_n-u)_-, r_n),\\
& \inf_{B_{n+1}} (u-m_n)\geq \frac{\sigma_\delta}{L^{1+\delta}}\dashint_{A_n}(u-m_n)\, dx -\frac{C_\delta}{L^\delta}\,{\rm Tail}((u-m_n)_-, r_n),
\end{split}
 \]
so that, adding up, we get
 \[
 \begin{split}
 M_n-m_n&-{\rm osc}(u, r_{n+1})=\inf_{B_{n+1}} (M_n-u)+\inf_{B_{n+1}}(u-m_n)\\
 &\geq \frac{\sigma_\delta}{L^{1+\delta}}(M_n-m_n)- \frac{C_\delta}{L^\delta}\big({\rm Tail}((M_n-u)_-, r_n)+{\rm Tail}((u-m_n)_-, r_n)\big).
 \end{split}
 \]
 Rearranging, we get
 \[
 {\rm osc}(u, r_{n+1})\leq \left(1-\frac{\sigma_\delta}{L^{1+\delta}}\right)(M_n-m_n)+\frac{C_\delta}{L^\delta}\big({\rm Tail}((M_n-u)_-, r_n)+{\rm Tail}((u-m_n)_-, r_n)\big).
 \]
 The tail terms can be estimated by strong induction as in the proof of \cite[Theorem 5.4]{IMS}, giving for $\alpha<2s$,
 \[
 {\rm Tail}((M_n-u)_-, r_n)\leq C\big(k \,S(\alpha)\, r_n^\alpha+ r_n^{2s}\big),
 \]
 where
 \[
 S(\alpha)=\sum_{h=0}^{+\infty}\frac{4^{\alpha h}-1}{4^{2sh}}=\frac{1}{1-4^{\alpha-2s}}-\frac{1}{1-4^{-2s}}\leq C\alpha.
 \]
 Therefore, being $r_n<1$, for $\alpha<2s$ we get, by the induction hypothesis and  $r_{n}=4 r_{n+1}$,
 \begin{equation}
 \label{osc2}
 \begin{split}
 {\rm osc}(u, r_{n+1})&\leq \left(1-\frac{\sigma_\delta}{L^{1+\delta}}\right)(M_n-m_n)+\frac{C_\delta}{L^\delta}\left(k\,\alpha\, r_{n}^\alpha +r_n^{2s}\right)\\
& \leq k\, \left[4^\alpha\left(1-\frac{\sigma_\delta}{L^{1+\delta}}\right) +\frac{C_\delta}{L^\delta}\,\alpha\right]r_{n+1}^\alpha+ \frac{C_\delta}{L^\delta}\,r_{n+1}^{\alpha}.
 \end{split}
 \end{equation}
 We start choosing $\alpha=a_\delta/L^{1+\delta}$ such that 
\begin{equation}
\label{poi}
4^\alpha\left(1-\frac{\sigma_\delta}{L^{1+\delta}}\right)\leq 1-\frac{\sigma_\delta}{2\,L^{1+\delta}}, 
\end{equation}
i.e.
\begin{equation}
\label{calpha}
\alpha\leq \log_4\left(1+\frac{\sigma_\delta}{2\,L^{1+\delta}-2\sigma_\delta}\right).
\end{equation}
Observe that for $L\geq 1$, $\frac 1 2\geq \sigma_\delta$ (which we may assume), there exists a constant $a_\delta$ such that
the previous inequality is satisfied for all $\alpha\leq  a_\delta/L^{1+\delta}$, thus giving \eqref{poi}. Eventually choosing a smaller $a_\delta$ (depending only on $\delta$), it also holds
\[
4^\alpha\left(1-\frac{\sigma_\delta}{L^{1+\delta}}\right) +\frac{C_\delta}{L^\delta}\,\alpha\leq
 1-\frac{\sigma_\delta}{2\,L^{1+\delta}}+\frac{C_\delta\, a_\delta}{L^{1+2\delta}}\leq 1-\frac{\sigma_\delta}{4\,L^{1+\delta}},
 \]
 so that for $\alpha=a_\delta/L^{1+\delta}$, \eqref{osc2} becomes
 \[
  {\rm osc}(u, r_{n+1})\leq k \left(1-\frac{\sigma_\delta}{4\,L^{1+\delta}}\right)r_{n+1}^\alpha+ \frac{C_\delta}{L^\delta}\,r_{n+1}^{\alpha}.
  \]
  Finally we choose
  \[
  k=L\,\frac{4\,C_\delta}{\sigma_\delta},
  \]
  to get 
  \[
  {\rm osc}(u, r_{n+1})\leq k \,r_{n+1}^\alpha,
  \]
 which ensures that $M_{n+1}$ and $m_{n+1}$ can be constructed satisfying \eqref{ind}.
\end{proof}

\begin{remark}
It is worth  outlining the dependance of $a_\delta$ on $s$ as $s\to 1$. As already pointed out, our exponent blows down as $s\to 1^-$. Indeed, using remark \ref{rems}, one gets $\sigma_\delta, \, C_\delta \simeq (1-s)$ in \eqref{temp2}. The only relevant place where $(1-s)$ is involved in the choice of the H\"older exponent is therefore \eqref{calpha}, which forces $a_\delta(s)=\tilde{a}(N, \delta) (1-s)$.
\end{remark}

\end{document}